\documentclass[11pt]{amsart}
\usepackage{geometry}                
\usepackage{graphicx}
\usepackage{amssymb}
\usepackage{epstopdf}

\newtheorem{theorem}{Theorem}
\newtheorem{remark}[theorem]{Remark}
\newtheorem{lemma}[theorem]{Lemma}
\newtheorem{corollary}[theorem]{Corollary}

\newcommand{\func}[1]{\operatorname{#1}}

\title{On contractible edges in convex decompositions}

\author{Ferran Hurtado \and Eduardo Rivera-Campo}
\thanks{Ferran Hurtado worked at Universitat Polit\`ecnica de Catalunya until his death in 2014}
\thanks{Eduardo Rivera-Campo, Universidad Aut\'onoma Metropolitana-Iztapalapa, {\tt erc@xanum.uam.mx}}
\thanks{Partially supported by Conacyt, M\'exico}
\date{}                                         

\begin{document}

\begin{abstract}

Let $\Pi$ be a convex decomposition of a set $P$ of $n\geq 3$ points in general position in the plane. If $\Pi$ consists of more than one polygon, then either $\Pi$ contains a deletable edge or $\Pi$ contains a contractible edge.
\end{abstract}

\maketitle
\section{Introduction}

Let $P$ be a set of $n\geq 3$ points in general position in the plane. A \emph{convex decomposition} of $P$ is a set $\Pi$ of convex polygons with vertices in $P$ and pairwise disjoint interiors such that their union is the convex hull $CH(P)$ of $P$ and that no point in $P$ lies in the interior of any polygon in $\Pi$.  A \emph{geometric graph} with vertex set $P$ is a graph $G$, drawn in the plane in such a way that every edge is a straight line segment with ends in $P$.  

Let $\Pi$ be a convex decomposition of $P$.  We denote by $G(\Pi)$ the \emph{skeleton graph} of $\Pi$, that is the plane geometric graph with vertex set $P$ in which the edges are the sides of all polygons in $\Pi$.  An edge $e$ of $\Pi$ is an \emph{interior edge} if $e$ is not an edge of the boundary of  $CH(P)$. 

An interior edge $e$ of $\Pi$ is \emph{deletable} if the geometric graph $G(\Pi) - e$, obtained from $G(\Pi)$ by deleting the edge $e$, is the skeleton graph of a convex decomposition of $P$. Neumann-Lara \emph{et al} \cite{NRU} proved that if a convex decomposition $\Pi$ of a set $P$ of $n$ points consists of more that $\frac{3n-2k}{2}$ polygons, where $k$ is the number of vertices of $CH(P)$, then $\Pi$ has at least one deletable edge.

 An interior edge $e=uv$ of $\Pi$ is \emph{contractible} from $u$ to $v$ if the geometric graph
 $G(\Pi) / \vec{uv}=(G(\Pi)- \{x_1u, x_2u, \ldots, x_mu, uv\}) + \{x_1v, x_2v, \ldots x_mv\}$ is a skeleton graph of a convex decomposition of $P \setminus \{u\}$, where $x_1,x_2, \ldots, x_m$ are the remaining vertices of $G(\Pi)$ which are adjacent to $u$.
 
A \emph{simple convex deformation} of $\Pi$ is a convex decomposition $\Pi'$ obtained from $\Pi$ by moving a single point $x$ along a  straight line segment, together with all the edges incident with $x$, in such a way that at each stage we have a convex decomposition of the corresponding set of points. Deformations of plane graphs have been studied by several authors, both theoretically  and algorithmically, see for instance \cite{C, GS, T} and \cite{BHL, BlPS, GHS}, respectively.

Let $P_1$ and $P_2$ be sets of $n \geq 3$  points in general position in the plane. A convex decomposition  $\Pi_1$ of $P_1$ and a convex decomposition $\Pi_2$ of $P_2$ are \emph{isomorphic} if there is an isomorphism of $G(\Pi_1)$ onto $G(\Pi_2)$, as abstract plane graphs, such that the boundaries of $CH(P_1)$ and $CH(P_2)$ correspond to each other with the same orientation. 

Thomassen \cite{T} proved that if $\Pi_1$ and $\Pi_2$ are \emph{isomorphic} convex decompositions, then  $\Pi_2$ can be obtained from  $\Pi_1$ by a finite sequence of simple convex deformations. As a tool, Thomassen proved that if $\Pi$ is a convex decomposition with at least two polygons, then there is an isomorphic convex decomposition $\Pi'$ that can be obtained from $\Pi$  by a finite number of simple convex deformations that preserve the boundary and such that $\Pi'$ contains either a deletable edge or a contractible edge. In this note we prove that every convex decomposition $\Pi$ with at least two polygons contains an edge which is deletable or contractible. Furthermore, if $P$ contains at least one interior point, then $\Pi$ contains a contractible edge.

\section{Preliminary results}

Let $\Pi $ be a convex decomposition of $P$ containing no deletable edges.
For every interior edge $e$ of $G\left( \Pi \right) $, the graph $G\left(\Pi \right) -e$ has an internal face $Q_e$ which is not convex and at least one end of $e$ is a reflex vertex of $Q_e$.

We define an abstract directed graph $\overrightarrow{G\left( \Pi \right) }$
with vertex set $P$ in which $\overrightarrow{uv}$ $\in $ $A\left( 
\overrightarrow{G\left( \Pi \right) }\right) $ if and only if $u$ is a
reflex vertex of $Q_{uv}$. Notice that for each interior edge $uv$ of $%
G\left( \Pi \right) $, the directed graph $\overrightarrow{G\left( \Pi
\right) }$ contains at least one of the arcs $\overrightarrow{uv}$ and $%
\overrightarrow{vu}$ (see Fig. \ref{fig:CD1}).

\begin{figure}
\begin{center}
 \includegraphics[width=.9\textwidth]{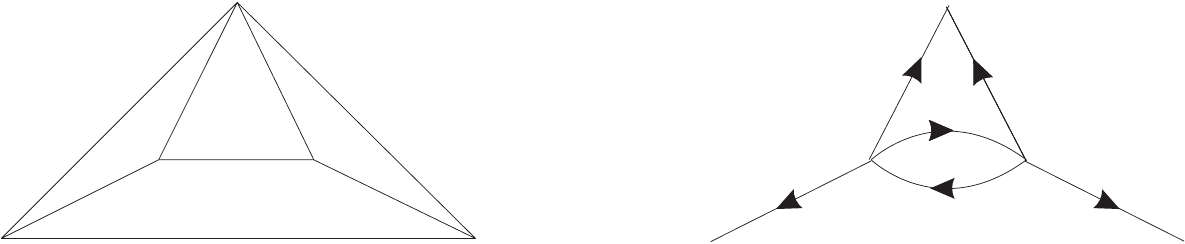}
 \caption{A Convex partition $\Pi$ and the corresponding directed graph $\protect\overrightarrow{G\left( \Pi \right) }$. }
  \label{fig:CD1} 
\end{center}
\end{figure}

\begin{remark}
\end{remark}

\begin{enumerate}

\item The outdegree of every vertex $u$ of $\overrightarrow{G\left( \Pi \right) 
}$ is at most $3$.

\item The outdegree of every vertex $u$ in the boundary of $CH\left( P\right) $
is $0$.

\item An interior vertex $u$ of $\Pi $ has outdegree 3 in $\overrightarrow{%
G\left( \Pi \right) }$ if and only if $u$ has degree 3 in $G\left( \Pi
\right) $.

\item If $\overrightarrow{uv},\overrightarrow{uw}$ $\in $ $A\left( 
\overrightarrow{G\left( \Pi \right) }\right) $, then $uv$ and $uw$ lie in a
common face of $G\left( \Pi \right) $.\\

\end{enumerate}

For two points $\alpha $ and $\beta $ in the plane, we denote by $r\left(
\alpha \beta \right) $ the ray, with origin $\alpha $, that contains the
segment $\alpha \beta $.

\begin{lemma}
\label{nocontr}
An edge $uv$ of $\Pi $ is not contractible from $u$ to $v$ if and only if
there are edges $yx$ and $xu$, lying in a common face of $G\left( \Pi
\right) $ that contains vertex $u$, such that the ray $r\left( yx\right) $
meets the edge $uv$ at point $u_{t}$, with $u\neq u_{t}\neq v$, and that the
triangular region defined by $x$, $u_{t}$ and $u$ contains no point of $P$
in its interior.
\end{lemma}

\begin{proof}
It is easy to see that the existence of such edges $yx$ and $xu$ implies that the edge $uv$ cannot be contracted from $u$ to $v$; we proceed to prove the sufficiency part of the lemma. Let $uv$ be an interior edge of $\Pi $ with $u$ not lying in the boundary of 
$CH\left( \Pi \right) $ and let $x_{1},x_{2},\ldots ,x_{m}$ be the remaining
vertices of $G\left( \Pi \right) $ which are adjacent to $u$. We contract
the edge $uv$ in a continuous way as follows: Slide the point $u$ along the
ray $r\left( uv\right) $, together with the edges $x_{1}u,x_{2}u,\ldots
,x_{m}u$ (see Fig. \ref{fig:CD2}).

\begin{figure}
\begin{center}
 \includegraphics[width=.9\textwidth]{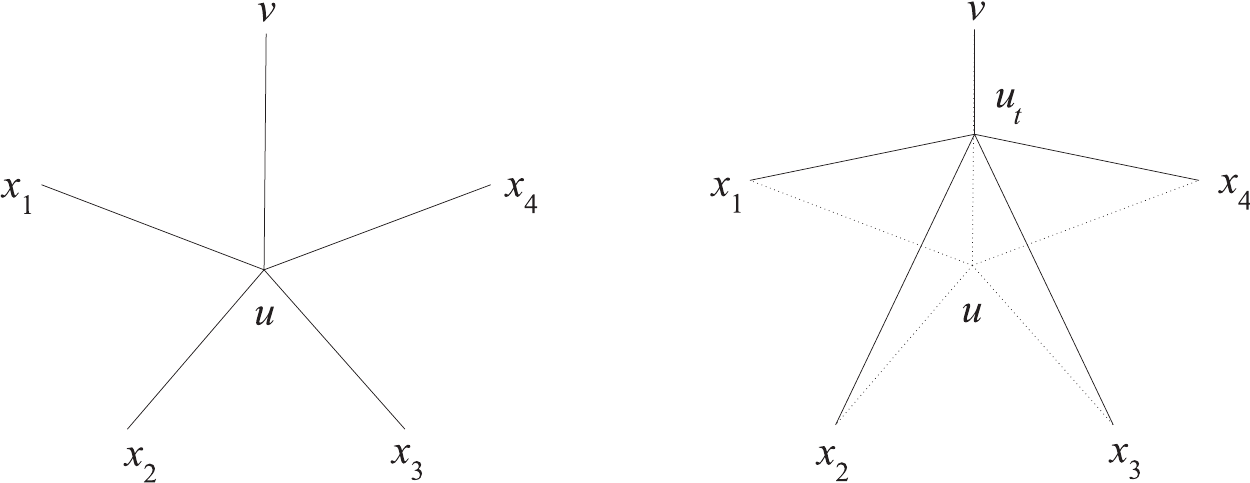}
 \caption{Contracting an edge $uv$ continuously.}
  \label{fig:CD2} 
\end{center}
\end{figure}

If $uv$ is not contractible from $u$ to $v$, then either the transformed
graph $T\left( G\left( \Pi \right) \right) $ becomes non planar or one of
its faces becomes non convex. This implies that we must reach a point $u_{t}=
$ $u+t\left( v-u\right) $, with $0<t<1$, such that there are two edges $
yx_{i}$ and $x_{i}u_{t}$ lying in a common face, which become collinear in $
T\left( G\left( \Pi \right) \right) $ (see Fig. \ref{fig:CD3}). 

Notice that two or more different pairs of edges 
$yx_{i}$,  $x_{i}u_{t}$ and  $y'x_{j}$, $x_{j}u_{t}$ may become collinear simultaneusly;
in such a case we may choose any of those pairs and proceed with the proof.

\begin{figure}
\begin{center}
 \includegraphics[width=.9\textwidth]{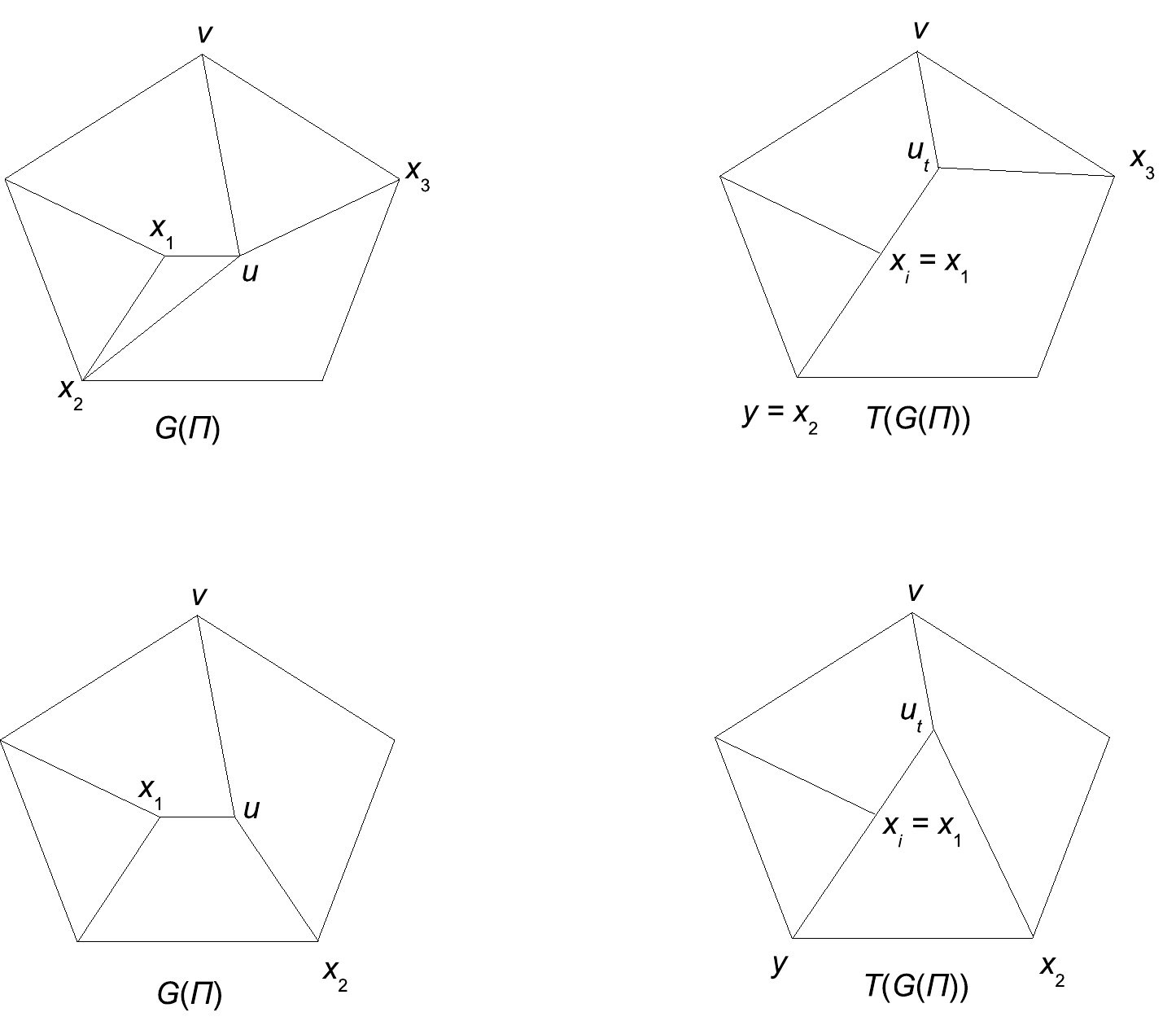}
 \caption{Edges $yx$ and $xu_t$ become collinear. }
  \label{fig:CD3} 
\end{center}
\end{figure}

The triangular region defined by $x_i$, $u_t$ and $u$ is the region swept by
the edge $x_iu_s$, $0 \leq s \leq t$ and therefore it contains no point of $P$ in its interior.
The lemma follows since the edges $yx_i$ and $x_iu$ lie in a common face of $%
G\left( \Pi \right) $ and the ray $r\left( yx_i\right) $ meets the edge $uv$
at the point $u_t$. 
\end{proof}

Let $N$ denote the set of arcs $\overrightarrow{uv}$ of $\overrightarrow{%
G\left( \Pi \right) }$ such that the edge $uv$ is not contractible from $u$
to $v$ in $\Pi $. For each $\overrightarrow{uv}\in N$ let $y=y_{uv}$, $%
x=x_{uv}$ and $u_t$ be as in Lemma \ref{nocontr}. Since the edges $y_{uv}x_{uv}$ and $%
x_{uv}u$ lie in a common face of $G\left( \Pi \right) $ and the triangular
region, defined by $x_{uv}$, $u_t$ and $u$, contains no point of $P$ in its
interior, the geometric graph $G\left( \Pi \right) -x_{uv}u$ contains a face $%
Q_{x_{uv}u}$ in which $x_{uv}$ is a reflex vertex and therefore $\overrightarrow{x_{uv}u}%
\in A\left( \overrightarrow{G\left( \Pi \right) }\right) $. This defines a
function $f:N\longrightarrow A\left( \overrightarrow{G\left( \Pi \right) }%
\right) $ given by $f\left( \overrightarrow{uv}\right) =\overrightarrow{%
x_{uv}u}$.

Notice that the arcs $f\left( \overrightarrow{uv}\right) $ and $%
\overrightarrow{uv}$ form a directed path in $\overrightarrow{G\left( \Pi
\right) }$ with length 2 and middle vertex $u$. This implies that if $%
f\left( \overrightarrow{u_1v_1}\right) =f\left( \overrightarrow{u_2v_2}%
\right) $, then $u_1=u_2$. Moreover, if $uv_1,uv_2$ and $uv_3$ are distinct
arcs such that $f\left( \overrightarrow{uv_1}\right) =f\left( 
\overrightarrow{uv_2}\right) =f\left( \overrightarrow{uv_3}\right) =%
\overrightarrow{xu}$, then $u$ is adjacent in $G\left( \Pi \right) $ to $%
v_1,v_2,v_3$ and to $x$, which is not possible by Remark 1, since $u$ has
outdegree 3 in $\overrightarrow{G\left( \Pi \right) }$. It follows that
there are no three arcs in $N$ with the same image under the function $f$
and therefore $\left|\func{Im}\left( f\right)\right|=\left|N\right|-\left|U\right|$, 
where $U$ is the set of points $u$ of $P$ for which
there is a pair of arcs $\overrightarrow{uv},\overrightarrow{uw}\in N$ such
that $f\left( \overrightarrow{uv}\right) =f\left( \overrightarrow{uw}\right) 
$.

\newpage

\begin{figure}
\begin{center}
 \includegraphics[width=.6\textwidth]{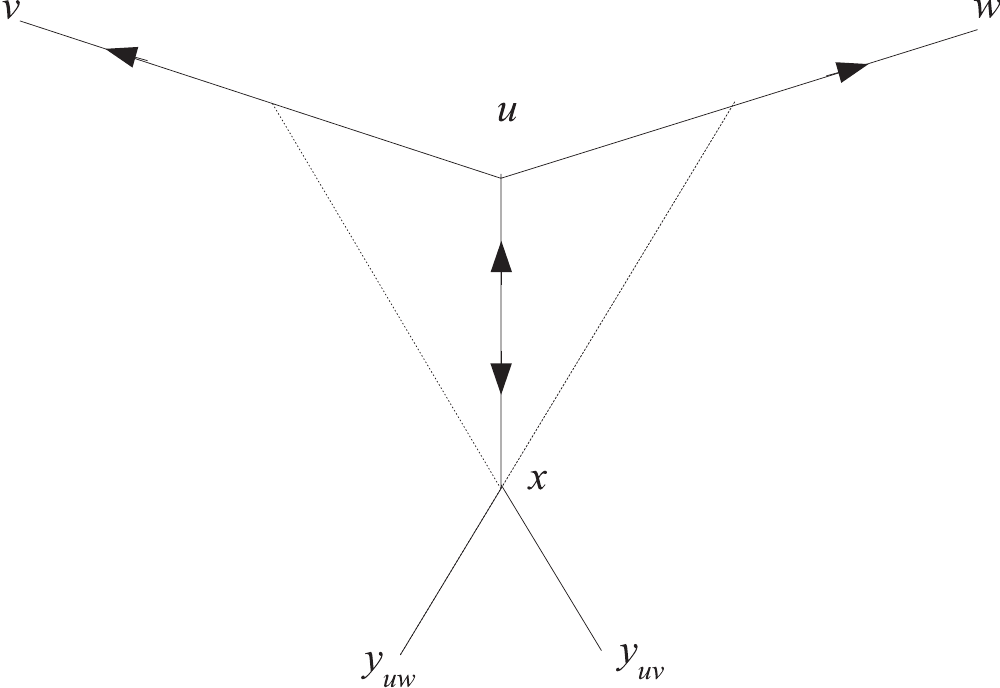}
 \caption{$f(uv) = f(uw) = \protect\overrightarrow{xu}$}
  \label{fig:CD4} 
\end{center}
\end{figure}

\begin{lemma}
\label{func}
Let $\Pi $ be a convex decomposition of $P$ with no deletable edges. If $U \neq \emptyset$, then there
is a function $g:U\rightarrow A\left( \overrightarrow{G\left( \Pi\right) }%
\right) $ such that for each $u\in U$, $g\left( u\right) $ is not in the
image of the function $f$.
\end{lemma}

\begin{proof}
Let $u\in U$ and let $v,w$ and $x=x_{uv}=x_{uw}$ be points in $P$ such that $%
f\left( \overrightarrow{uv}\right) =f\left( \overrightarrow{uw}\right) =%
\overrightarrow{xu}$. If $u$ has degree larger than 3 in $G\left( \Pi\right) $%
, let $z\notin \left\{ v,w,x\right\} $ be such that $uz$ is an edge of $%
G\left( \Pi\right) $. By Remark 1, the outdegree of $u$ in $\overrightarrow{%
G\left( \Pi\right) }$ is at most 2, therefore $\overrightarrow{uz}$ is not an
arc of $\overrightarrow{G\left( \Pi\right) }$. It follows that $%
\overrightarrow{zu}$ must be an arc of $\overrightarrow{G\left( \Pi\right) }$.
In this case $g\left( u\right) =\overrightarrow{zu}$ $\notin \func{Im}\left(
f\right) $ since $z \neq x$ and $\overrightarrow{xu}$ is the unique arc in $Im(f)$ that ends at $u$.


If $u$ has degree 3 in $G\left( \Pi \right) $, then $u$ has outdegree 3 in $%
\overrightarrow{G\left( \Pi\right) }$, by Remark 1 and therefore $%
\overrightarrow{ux}$ is an arc $\overrightarrow{G\left( \Pi \right) }$. We
claim that in this case $g\left( u\right) =\overrightarrow{ux}$ $\notin 
\func{Im}\left( f\right) $. Let $l_{ux}$ denote the line containing the edge 
$ux$, and let $y_{uv}$ and $y_{uw}$ be points in $P$ and such that the rays $%
r\left( y_{uv}x\right) $ and $r\left( y_{uw}x\right) $ intersect the edges $%
uv$ and $uw$, respectively.

Without loss of generality we assume that $l_{ux}$ is a vertical line such
that $v$ and $y_{uw}$ lie to the left of $l_{ux}$ and $w$ and $y_{uv}$ lie
to the right of $l_{ux}$ (see Fig. \ref{fig:CD4}). Clearly the angles $\measuredangle
y_{uv}xu$ and $\measuredangle y_{uw}xu$ are smaller than $\pi $, it is easy
to see that $\measuredangle y_{uw}xy_{uv}$ is also smaller than $\pi $.

Therefore if $xz$ is an edge of $\Pi $ with $z\notin \left\{
u,y_{uv},y_{uw}\right\} $, then $\overrightarrow{xz}$ is not an arc of $%
\overrightarrow{G\left( \Pi \right) }$. This implies that if $%
\overrightarrow{ux}\in \func{Im}\left( f\right) $, then $\overrightarrow{ux}%
=f\left( \overrightarrow{xy_{uv}}\right) $ or $\overrightarrow{ux}=f\left( 
\overrightarrow{xy_{uw}}\right) $ since $f\left( \overrightarrow{a}\right) $
and $\overrightarrow{a}$ form a directed path of length 2 for each arc $%
\overrightarrow{a}\in N$.

Suppose $\overrightarrow{ux}=f\left( \overrightarrow{xy_{uv}}\right) $. By
the definition of $f$, there is an edge $y_{xy_{uv}}u$ such that the ray $%
r\left( y_{xy_{uv}}u\right) $ intersects the edge $xy_{uv}$. Since $v$ and $%
w $ are the only vertices different from $x$ which are adjacent to $u$ in $%
G\left( \Pi \right) $, one of them must be the vertex $y_{xy_{uv}}$. Since
both edges $uw$ and $xy_{uv}$ lie in the right halfplane defined by $l_{ux}$
then $r\left( wu\right) $ cannot intersect the edge $xy_{uv}$ and therefore $%
y_{xy_{uv}}\neq w$. Finally, since $r\left( y_{uv}x\right) $ intersects the
edge $uv$, $r\left( vu\right) $ cannot intersect the edge $xy_{uv}$.
Therefore $\overrightarrow{ux}\neq f\left( \overrightarrow{xy_{uv}}\right) $; analogously $\overrightarrow{ux}\neq f\left( \overrightarrow{xy_{uw}}%
\right) $.
\end{proof}

\section{Main results}

In this section we prove our main results.

\begin{theorem}
\label{either}
Let $P$ be a set of points in general position in the plane. If $\Pi $ is a
convex decomposition of $P$ consisting of more than one polygon, then either 
$\Pi $ contains a deletable edge or $\Pi $ contains a contractible edge.
\end{theorem}

\begin{proof}
Assume the result is false and $\Pi $ contains no deletable edges and no
contractible edges. Define the directed graph $\overrightarrow{G\left( \Pi
\right) }$ as in the previous section, notice that $A\left( \overrightarrow{%
G\left( \Pi \right) }\right) \neq \emptyset $ since $\Pi $ contains at least
two polygons. Since $\Pi $ contains no contractible edges, $N=A\left( 
\overrightarrow{G\left( \Pi \right) }\right) $.

Let $B=B\left( \overrightarrow{G\left( \Pi\right) }\right) $ be the set of
arcs of $\overrightarrow{G\left( \Pi\right) }$ of the form $\overrightarrow{uw}
$, with $w$ in the boundary of $CH\left( P\right) $, and let $%
\overrightarrow{uw}\in B$. By Remark 1, $w$ has outdegree 0 in $%
\overrightarrow{G\left( \Pi\right) }$ which implies $\overrightarrow{uw}$ $%
\notin $ $\func{Im}\left( f\right) $.

If $U=\emptyset$, then $\func{Im}\left( f\right) \subset A\left( \overrightarrow{G\left( \Pi\right) }
\right) \backslash B$, therefore $\left| N\right| = \left| \func{Im}\left( f\right) \right| \leq \left| A\left( \overrightarrow{G\left( \Pi\right) } \right) \backslash B \right| \leq  \left| A\left( \overrightarrow{G\left( \Pi\right) } \right) \right| -3$, which is not possible since $\Pi $ contains no deletable edges and $\left| B\right| \geq 3$.

And if $U\neq \emptyset$, by Lemma \ref{func} no arc in $\func{Im}\left( g\right) $ lies in $\func{Im}\left(
f\right) $, therefore $\func{Im}\left( f\right) \subset A\left( \overrightarrow{G\left( \Pi\right) }
\right) \backslash \left( \func{Im}\left( g\right) \cup B\right) $. In this case $\left| \func{%
Im}\left( f\right) \right| \leq \left| A\left( \overrightarrow{G\left(
\Pi\right) }\right) \right| -\left| \func{Im}\left( g\right) \right| -\left|
B\right| $, since $g\left( u\right) \notin B$. This is a contradiction since $A\left( \overrightarrow{G\left(
\Pi\right) }\right) =N$, $\left| \func{Im}\left( g\right) \right| =\left|
U\right| $, $\left| B\right| \geq 3$ and $\left| \func{Im}\left( f\right)
\right| =\left| N\right| -\left| U\right| $.
\end{proof}

\begin{corollary}
\label{onecontr}
Let $\Pi $ be a convex decomposition of a set of points $P$ in general
position in the plane. If $P$ contains at least one interior point, then $%
\Pi $ contains at least one contractible edge.
\end{corollary}

\begin{proof}
Let $\Pi ^{\prime }$ be a convex decomposition of $P$ obtained from $\Pi $
by removing deletable edges, one at a time, until no such edges remain, and
let $\overrightarrow{G\left( \Pi ^{\prime }\right) }$ be the corresponding
directed abstract graph. Since $P$ contains an interior point, $\Pi ^{\prime
}$ contains at least one interior edge.

By Theorem \ref{either}, there is an arc $\overrightarrow{uv}\in A\left( 
\overrightarrow{G\left( \Pi^{\prime }\right) }\right) $ such that $uv$ is
contractible from $u$ to $v$ in $\Pi ^{\prime }$. If $uv$ is not
contractible in $\Pi $, then by Lemma 1 there are edges $yx$ and $xu$ lying
in a common face of $G\left( \Pi \right) $ such that the ray $r\left(
yx\right) $ meets the edge $uv$ at an interior point $u_t$ and that the
triangular region $yu_tu$ contains no point of $P$ in its interior. This
implies that the geometric graph $G\left( \Pi \right) -xu$ contains a face $%
Q_x$ in which $x$ is a reflex vertex and therefore $xu$ is not deletable in $%
\Pi $ and $\overrightarrow{xu}$ is an arc of $\overrightarrow{G\left(
\Pi\right) }$.

Let $R$ be the face of $G\left( \Pi \right) $ which contains both edges $yx$
and $xu$. Since $\Pi ^{\prime }$ is obtained from $\Pi $ by deleting edges
but no points, then there is a face $R^{\prime }$ of $G\left( \Pi ^{\prime
}\right) $ which contains the edge $xu$ and the region bounded by $R$, let $%
y^{\prime }\in P$ be such that $y^{\prime }x$ is an edge of $R^{\prime }$.
Notice that $y^{\prime }\neq y$ otherwise $uv$ could not be a contractible
edge of $\Pi ^{\prime }$ because the ray $r\left( yx\right) $ meets the edge 
$uv$ at the point $u_{t}$ (Fig. \ref{fig:CD5}, left). Nevertheless, since the face $%
R^{\prime }$ contains the edge $xu$ and the region bounded by $R$, the ray $%
r\left( y^{\prime }x\right) $ also meets the edge $uv$ at an interior point $%
u_{t^{\prime }}$ (Fig. \ref{fig:CD5}, right) which again is a contradiction.%
\end{proof}

\begin{figure}
\begin{center}
 \includegraphics[width=.9\textwidth]{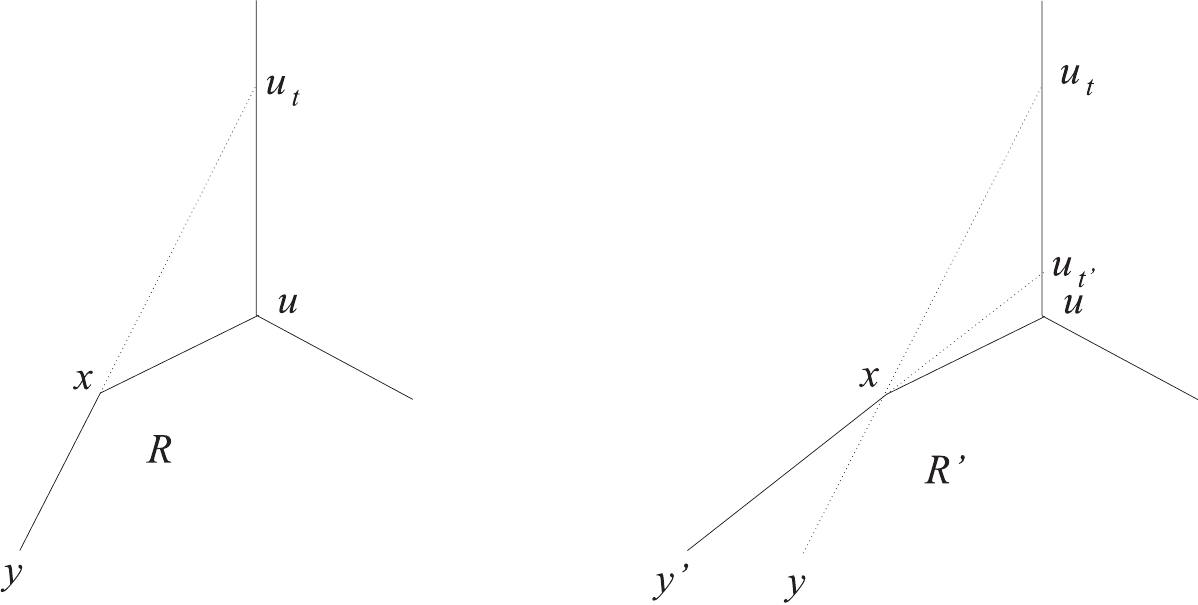}
 \caption{Left: Ray $r\left( yx\right) $ meets  edge 
$uv$ at the point $u_{t}$. Right: Ray $
r\left( y^{\prime }x\right) $ meets edge $uv$ at an interior point $
u_{t^{\prime }}$. }
  \label{fig:CD5} 
\end{center}
\end{figure}

\begin{corollary}
Let $\Pi $ be a convex decomposition of a set of points $P$ in general
position in the plane and $Q$ be the set of points in the boundary of $%
CH\left( P\right) $. There is a sequence $P=P_{0},P_{1},\ldots ,P_{m}=Q$ of
subsets of $P$, and a sequence $\Pi _{0},\Pi _{1},\ldots ,\Pi _{m}$ of
convex decompositions of $P_{0},P_{1},\ldots ,P_{m}$, respectively, such
that $\Pi _{0}=\Pi $, $\Pi _{m}$ consists of the boundary of $CH\left(
P\right) $ and for $i=0,1,\ldots ,k$, $\Pi _{i+1}$ is obtained from $\Pi _{i}
$ by contracting an edge and for $i=k+1,k+2,\ldots ,m-1$, $\Pi _{i+1}$ is
obtained from $\Pi _{i}$ by deleting an edge.\-
\end{corollary}

\begin{proof}
By Corollary \ref{onecontr} , if $P_i$ contains interior points, then $\Pi _i$ has a
contractible edge. If $P_i$ contains no interior points, then each interior
edge of $\Pi _i$ is a deletable edge.
\end{proof}

\end{document}